\documentclass{amsart}

\title[Degenerate direction]{Degenerate characteristic directions for maps tangent to the Identity}
\author{Liz Vivas}
\date{May 25th, 2011}

\usepackage[usenames,dvipsnames]{pstricks}
\usepackage{pst-grad} 
\usepackage{pst-plot} 

\newtheorem{theorem}{Theorem}
\newtheorem{lemma}{Lemma}
\newtheorem{proposition}{Proposition}

\theoremstyle{definition}

\theoremstyle{remark}
\newtheorem{remark}{Remark}

\newcommand{\NN}{\mathbb{N}}

\newcommand{\CC}{\mathbb{C}}
\newcommand{\PP}{\mathbb{P}}

\newcommand{\Ree}{\textrm{Re}}
\newcommand{\Arg}{\textrm{Arg}}
\newcommand{\Id}{\textrm{Id}}

\newcommand{\Ind}{\textrm{Ind}}

\newcommand{\tilf}{\tilde F}

\begin{document}

\bibliographystyle{plain}

\begin{abstract}
Let $F:(\CC^2,O) \to (\CC^2,O)$ be a germ tangent to the identity. Assume $F$ has a characteristic direction $[v]$. In \cite{Hak} Hakim gives conditions to guarantee the existence of an attracting basin to the origin along $[v]$, in the case of $[v]$ a non-degenerate characteristic direction. In this paper we give conditions to guarantee the existence of basins along $[v]$ in the case of $[v]$ a degenerate characteristic direction.

\end{abstract}

\maketitle

\section{Introduction}

In this paper we study the local dynamics of maps tangent to the identity. That is, we consider of germs of holomorphic self-maps $F: \CC^n \to \CC^n$ such that $F(O) = O$, where $O \in \CC^n$ is the origin and $dF(O) = Id$. 
When $n = 1$ the dynamics is described by the celebrated Leau-Fatou flower theorem. In the case of $n > 1$ recent progress has been made to understand the dynamics and significant results have been obtained (see, e.g., \cite{Ab-Tov},\cite{Ab2}, \cite{Br} \cite{Hak},\cite{We},\cite{Mo},\cite{Vi}). However, we are still far from understanding the complete picture.

We investigate conditions ensuring the existence of open attracting domains to the fixed point for maps tangent to the identity in dimension $2$. Open domains are related to characteristic directions (see later for definitions). Characteristic directions are in turn classified in three different types (Fuchsian, irregular and apparent; \cite{Ab-Tov}). Hakim has given necessary conditions to guarantee the existence of basins for characteristic directions that are non-degenerate with non-vanishing index (a particular type of Fuchsian direction). We generalize her result for all Fuchsian directions:

\begin{theorem}
Let $F$ be a germ of a holomorphic self-map of $(\CC^2,O)$ tangent to the identity. Assume $[v]$ is  a degenerate characteristic direction that is Fuchsian. If the real part of the inverse of the index $I(\tilde{F},\PP^1,[v])$ belongs to the region $R$ (see figure \ref{Region R}), then there exists an open basin $V$ attracted to the origin along $[v]$. If $F$ is an automorphism of $\CC^2$ then $\displaystyle{\Omega = \bigcup_{i\geq 0}F^{-i}(V)}$ is biholomorphic to $\CC^2$.
\end{theorem}

We also prove a result on the existence of basins for all irregular directions:
\begin{theorem}
Let $F$ be a germ of holomorphic self-map of $(\CC^2,O)$ tangent to the identity. Assume $[v]$ is  a degenerate characteristic direction that is irregular. Then there exists an open basin $V$ attracted to the origin along $[v]$. If $F$ is an automorphism of $\CC^2$ then $\displaystyle{\Omega = \bigcup_{i\geq 0}F^{-i}(V)}$ is biholomorphic to $\CC^2$.
\end{theorem}

We now introduce the definitions and explain how our results are related to what is already known.

Let $F \neq \Id$ be a germ of holomorphic self-map of $\CC^2$ fixing the origin and tangent to the identity. Let
$$F(z,w) = (z,w) + P_k(z,w) + P_{k+1}(z,w) + \ldots$$
be the homogeneous expansion of $F$ in series of homogeneous polynomials, where $\deg P_j = j$ (or $P_j \equiv 0$) and $P_k \neq 0$. We say (and fix from now on) the order $\nu(F)$ of $F$ is $k$.

A \textit{parabolic curve} for $F$ at the origin is an injective map $\phi: \bar\Delta
\to \CC^2$, where $\Delta = \{z \in\CC ; |z-1| <1\} $ satisfying the following properties:
\begin{itemize}
\item[(i)] $\phi$ is holomorphic in $\Delta$, continuous on $\bar\Delta$ and $\phi(0) = O$; 
\item[(ii)] $\phi(\Delta)$ is invariant under $F$, and $F^n|\phi(\Delta) \to O$ as $n \to \infty$ uniformly on $\Delta$.
\end{itemize}
Furthermore, if $[\phi(z)] \to [v] \in \PP^1$ as $z \to 0$, where $[\cdot]$ denotes the canonical projection of $\CC^2 \backslash {O}$ onto $\PP^1$, we say that $\phi$ is tangent to $[v]$ at the origin.

We say $[v] = [v_1 : v_2] \in \PP^1$ is a \textit{characteristic direction} for $F$ if there is $\lambda \in \CC$ such that 
$P_k(v_1,v_2) = \lambda(v_1,v_2)$. If $\lambda \neq 0$, we say that $[v]$ is \textit{nondegenerate}; 
otherwise, it is \textit{degenerate}. 

It is easy to see that we either have infinitely many characteristic directions or $k+1$ characteristic directions, if counted with multiplicities. In the former case we say the origin is \textit{dicritical}; this case has been studied by Brochero-Martinez (see \cite{Bro}).

Characteristic directions arise naturally in the study of maps tangent to the identity due to the following fact: if there exist parabolic curves tangent to a direction $[v]$ then this direction is necessarily characteristic (\cite{Hak}). 

Hakim (and \'Ecalle with his method of resurgence theory) proved the converse for nondegenerate characteristic directions:

\begin{theorem} (\cite{Ec}; \cite{Hak}). Let $F$ be a germ of holomorphic self-map of $\CC^2$ fixing the origin and tangent to the identity. Then for every nondegenerate characteristic direction $[v]$ of $F$ there are $k-1$ parabolic curves tangent to $[v]$ at the origin.
\end{theorem}

In order to give a condition on the existence of basins for a non-degenerate characteristic direction, Hakim defines an index (see \cite{Hak}) as follows:

Let $[v]=[1,u_o]$ be a nondegenerate characteristic direction. Write $P_k=(p_k,q_k)$ then we have $u_op_k(1,u_o) = q_k(1,u_o)$ and $p_k(1,u_o) \neq 0$. Define $r(u) = q_k(1,u) - up_k(1,u)$, we clearly have $r(u_o) = 0$. The Hakim index is defined to be
$$
i_H(F,[v]) = \frac{r'(u_o)}{p_k(1,u_o)}.
$$
With this definition Hakim proves:

\begin{theorem}(\cite{Hak})
Let $F$ be a germ of holomorphic self-map of $\CC^2$ fixing the origin and tangent to the identity. Let $[v]$ be a nondegenerate characteristic direction. Assume that 
$\Ree(i_H(F,[v])) > 0$. Then there exists an open domain, in which every point is attracted to the origin along a trajectory tangent to $[v]$.
\end{theorem}

Returning to the question of existence of parabolic curves, in the case of degenerate characteristic direction, analogous results to Theorem 3 have been proven by Abate, Tovena, Molino.

\begin{theorem}(\cite{Ab1})
Let $F$ be a germ of holomorphic self-map of $\CC^2$ tangent to the identity and such that the origin is an isolated fixed point. Then there exist (at least) $k - 1$ parabolic curves for $F$ at the origin.
\end{theorem}

In order to prove this result, Abate modifies the geometry of the ambient space via a finite number of blow-ups and also defines a residual index $\Ind(\tilde{F},S,p) \in \CC$, where $\tilde{F}$ is a holomorphic self-map of a complex 2-manifold $M$ which is the identity on a 1-dimensional submanifold $S$, and $p \in S$. 

In particular, Abate proves the following for those characteristic directions whose residual index is not a non-negative rational number. Later, Molino generalized this results for maps whose residual index does not vanish, under an extra assumption ($F$ regular along $[v]$).

\begin{theorem}(\cite{Ab1},\cite{Mo}) Let $F$ be a germ of holomorphic self-map of $(\CC^2,O)$ tangent to the identity and such that the origin is an isolated fixed point. Let $[v] \in \PP^1$ be a characteristic direction of $F$ and assume $F$ is regular along $[v]$ with $\Ind(\tilde{F},\PP,[v]) \neq 0$ (where $\tilde{F}$ is the blow-up of $F$ and $\PP$ is the exceptional divisor). Then there exist parabolic curves for $F$ tangent to $[v]$ at the origin.
\end{theorem}

Examples of basins along degenerate characteristic directions have been also shown recently (see \cite{Ab2},\cite{Vi}). Also, in a recent paper, Abate and Tovena \cite{Ab-Tov} studied the dynamics of the time $1$-map of homogeneous holomorphic vector fields. In their paper, they differentiate between characteristic directions:  \textit{Fuchsian}, similar than non degenerate characteristic directions with non-vanishing index);  \textit{apparent}, degenerate directions that have, for time $1$-maps of homogeneous vector fields, no attracting dynamics along them; and the rest, which they call \textit{irregular}. 

Our Theorems $1$ and $2$ are therefore analogues of Theorem $4$ for degenerate characteristic directions. 

The organization of the paper is as follows. In the next section we explain how to compute the index defined by Abate, and we explain the relationship between the index defined by Hakim and the index defined by Abate. We also explain how to classify the characteristic directions in apparent, Fuchsian and irregular directions. In Section 3 we prove the main lemmas used for the proof of the theorems. In Section 4 we prove Theorem $1$ and in Section 5 we prove Theorem $2$. Finally, in the last section we proved that the basins are all biholomorphic to $\CC^2$ when $F$ is an automorphism.

\textit{Acknowledgements}: The author would like to thank Marco Abate, Eric Bedford and Han Peters for enlightening discussions about this paper. Also, deep thanks to L\'aszl\'o Lempert who commented on an earlier draft.

\section{Background: Abate Index and Hakim Index}

We have already explained in the last section how to compute the Hakim Index. Let us write this index more explicitly, in terms of the expansion of $P_k=(p_k,q_k)$.

Let $F$ be our map tangent to the identity, and let $[v] = [1:0]$ be a characteristic direction (we can do this by conjugating $F$ with a rotation). Then $q_k(1,0) = 0$.

In terms of the expansion of $P_k$: $p_k(z,w) = \sum_{i=0}^k a_iz^{k-i}w^i = a_0z^k + a_{1}z^{k-1}w + \ldots + a_{k-1}zw^{k-1} + a_kw^k$ and $q_k(z,w) = \sum_{i=0}^k b_iz^{k-i}w^i = b_0z^k + b_{1}z^{k-1}w + \ldots + b_{k-1}zw^{k-1} + b_kw^k$, we have $q_k(1,0) = b_0 = 0$.

By definition $[v]$ is a non-degenerate characteristic direction if $p_k(1,0) = a_0 \neq 0$ and degenerate if $p_k(1,0) = a_0 = 0$. 

Assume $[v]$ is a non-degenerate characteristic direction. Then we compute the index defined by Hakim in terms of the coefficients $a_i, b_i$. A simple computation shows:
$$
i_{H}(F,[v]) = \frac{b_1-a_0}{a_0}.
$$
%

%

Now, let us explain how to compute the index defined by Abate, which is defined not only for non-degenerate characteristic directions, but for a larger class of characteristic directions.

As was mentioned above, the index defined by Abate is actually defined for $\tilde{F}$, which is the blow-up of $F$ at the origin. In the blow up of $F$, our characteristic direction $[1:0]$ becomes a point in the exceptional divisor. 

In the chart $(z,u) = (z,w/z)$ for the blow-up of $F$ we have:

\begin{align*}
\tilde{F}(z,u) = \left(z + p_k(z,zu) + p_{k+1}(z,zu) + \ldots, \frac{zu + q_k(z,zu) + q_{k+1}(z,zu) + \ldots}{z + p_k(z,zu) + p_{k+1}(z,zu) + \ldots}\right)\\
= \left(z + z^kp_k(1,u) + z^{k+1}p_{k+1}(1,u) + \ldots, \frac{zu + z^kq_k(1,u) + z^{k+1}q_{k+1}(1,u) + \ldots}{z + z^kp_k(1,u) + z^{k+1}p_{k+1}(1,u) + \ldots}\right)\\
= \left(z + z^kp_k(1,u) + z^{k+1}p_{k+1}(1,u) + \ldots,u + z^{k-1}[q_k(1,u)-up_k(1,u)] + O(z^{k})\right).
\end{align*}
We obtain for $\tilde{F} = (\tilde F_1, \tilde F_2)$:
\begin{align}
\tilde F_1(z,u) &= z + z^k\left[p_k(1,u) + zp_{k+1}(1,u) + O(z^2)\right] \label{blowF}\\
\tilde F_2(z,u) &= u + z^{k-1}\left[r(u) + O(z)\right]  ,
\end{align}
where $r(u) = q_k(1,u) - up_k(1,u)$.
We have $\tilde F (0,u) = (0,u)$, which means that $\tilde F$ fixes the exceptional divisor. Also, the fact that $[1:0]$ is a characteristic direction for $F$ means that $r(0) =0$ for $\tilf$. The origin is dicritical if $r(u) \equiv 0$. Assume from now on the origin is not dicritical, i.e. $r(u) \not\equiv 0$. We will make a remark at the end of the section about the dicritical case.

We now explain the distinction between characteristic directions as defined by Abate and Tovena on \cite{Ab-Tov}.
If $m$ is the lowest order degree term of $p_k(1,u)$ and $n$ is the lowest order degree term of $r(u)$, then we have:
\begin{itemize}
\item[a.] If $1+m = n$, we say $[1:0]$ is a \textit{Fuchsian} characteristic direction. 
\item[b.] If $1+m < n$, we say $[1:0]$ is an \textit{irregular} characteristic direction.
\item[c.] If $1+m > n$ or $m = \infty$, we say $[1:0]$ is an \textit{apparent} characteristic direction. 
\end{itemize}

Abate defines the index of $F$ at the characteristic direction $v = [1:0]$ (which is the point $u=0$) as follows
$$
\Ind(\tilf,\PP^1,[v]) = Res_0(k(u)) 
$$
where 
$$
k(u) = \lim_{z\to 0}\frac{\tilf_1(z,u)-z}{z(\tilf_2(z,u)-u)}.
$$

In terms of $p_k$ and $q_k$ then:
\begin{align*}
k(u) = \lim_{z\to 0}\frac{z^kp_k(1,u) + O(z^{k+1})}{z(z^{k-1}r(u) +O(z^k))} = \frac{p_k(1,u)}{r(u)}.
\end{align*}

Therefore
\begin{align*}
\Ind(\tilf,\PP^1,[v]) = Res_0\frac{p_k(1,u)}{r(u)},
\end{align*}

In term of the coefficients of $p_k$ and $q_k$:
$$
p_k(1,u) = a_0 + a_1u + \ldots + a_ku^k
$$
and
$$
r(u) = (b_1-a_0)u + (b_2-a_1)u^2 + \ldots + (b_k-a_{k-1})u^k-a_{k}u^{k+1}
$$

We define
\begin{align*}
m:=& \min\{h \in \NN, a_h \neq 0\}, \\
n:=& \min\{j \in \NN, b_{j} - a_{j-1} \neq 0\}. 
\end{align*}

In the case of $p_k \equiv 0$, then we say $m = \infty$. Since $r(u)$ does not vanish identically, we have $n < \infty$.

Back to the definition of the index:
\begin{align*}
\Ind(\tilf,\PP^1,[v]) = Res_0\frac{p_k(1,u)}{r(u)} = Res_0\frac{a_mu^m + O(u^{m+1})}{c_nu^n + O(u^{n})} ,
\end{align*}

For each of the three cases above we compute the index:

\begin{itemize}
\item[(a)] \emph{$[1:0]$ is a Fuchsian direction i.e. $m+1=n$.}
\begin{itemize}
\item[(a.1)] $m=0, n=1$. 
If $m=0$, i.e. $a_0 \neq 0$, so we have $(1,0)$ is a non-degenerate characteristic direction (since $p_k(1,0) = a_0$). 
\begin{align*}
\Ind(\tilf,\PP^1,[v]) = \frac{a_0}{b_1-a_0} = \frac{1}{i_H(F,[v])}.
\end{align*}
We obtain that the index defined by Abate is the inverse of the index defined by Hakim, if the latest is not $0$. 
\\
\item[(a.2)]$m>0, n = m+1$.
\begin{align*}
\Ind(\tilf,\PP^1,[v]) = Res_0\frac{a_1u +a_2u^2+\ldots}{(b_1-a_0)u + (b_2-a_1)u^2 + \ldots}.
\end{align*}
\begin{align*}
\Ind(\tilf,\PP^1,[v]) = \frac{a_{n-1}}{b_{n}-a_{n-1}} = \frac{a_{m}}{b_{n}-a_{m}}.
\end{align*}
\end{itemize}
\item[(b)] \emph{$[1:0]$ is an irregular direction i.e. $m+1<n$.}
\begin{itemize}
\item[(b.1)] $m=0, n>1$.
We have $(1,0)$ is a non-degenerate characteristic direction and, plugging in the definition of index by Hakim we obtain $i_H(F,[v]) = 0$. However, the Abate index is not necessarily $0$.
\begin{align*}
\Ind(\tilf,\PP^1,[v]) &= Res_0\frac{a_0+a_1u +a_2u^2+\ldots}{(b_2-a_1)u^2 + \ldots}\\
&= \frac{a_{n-1}}{b_n-a_{n-1}}
\end{align*} 
\item[(b.2)] $m>0, n>m+1$.
In this case $(1,0)$ is degenerate and we have the same than above:
\begin{align*}
\Ind(\tilf,\PP^1,[v]) = \frac{a_{n-1}}{b_{n}-a_{n-1}}.
\end{align*}
\end{itemize}
\item[(c)] \emph{$[1:0]$ is an apparent direction i.e. $n<m+1$.}
\begin{itemize}
\item[(b.1)] $m<\infty$. 
In this case we have that $u=0$ is not a pole of $k(u)$ and we have:
\begin{align*}
\Ind(\tilf,\PP^1,[v]) = 0.
\end{align*}

\item[(b.2)] $m= \infty$.
Also $(1,0)$ is a degenerate characteristic direction. In the definition of index we obtain:
\begin{align*}
\Ind(\tilf,\PP^1,[v]) = Res_0 0 = 0.
\end{align*}

\end{itemize}

\end{itemize}

Let us summarize what we have in a table:
\begin{center}
\begin{tabular}{c|c|c|c|c|c|c} \cline{2-7}& m $\backslash$  n & 1 & 2 & 3 & 4 & \ldots   \\ \cline{2-7} Non-degenerate $\rightarrow$ &0 & Fuchsian  & Irregular & Irregular & Irregular&Irregular   \\ \cline{2-7} & 1 & Apparent & Fuchsian &Irregular &Irregular&Irregular   \\ \cline{2-7} & 2 &Apparent &Apparent  &Fuchsian  & Irregular& Irregular\\ \cline{2-7} & 3 & Apparent & Apparent& Apparent&Fuchsian  &Irregular \\ \cline{2-7} & \ldots & & & & &
\end{tabular}
\end{center}

With this notation we have the following theorems:

\begin{theorem}\label{Fuchsian}
Let $F$ be a germ of holomorphic self-map of $\CC^2$ tangent to the identity. Assume $[v]$ is a degenerate characteristic direction that is Fuchsian. If $\Ind(\tilf, \PP^1, [v]) \in R$, where
$$R = \left\{\zeta \in \CC, \Re(\zeta) > -\frac{m}{k-1}, \left|\zeta-\frac{m+1-m/(k-1)}{2}\right| > \frac{m+1+m/(k-1)}{2}\right\} \subset \CC$$
then there exists an open basin attracted to the origin along $[v]$.
\end{theorem}
\begin{center}
\begin{figure}[h]

{
\begin{pspicture}(0,-4.0051847)(6.366173,4.0051847)
\psframe[linewidth=0.0020,linestyle=dotted,dotsep=0.16cm,dimen=outer,fillstyle=vlines*,hatchwidth=0.04,hatchangle=0.0](6.366173,3.5399544)(1.4261726,-3.9200454)
\rput(2.0069604,0.0){\rput{0.09976467}(0.0,-0.001741221){\psaxes[linewidth=0.038,tickstyle=bottom,labels=none,ticksize=0.0358cm,showorigin=false](0,0)(-2,-4)(4,4)}}
\rput{-0.11088063}(0.0,0.005663013){\pscircle[linewidth=0.04,dimen=outer,fillstyle=solid](2.9263704,-0.10087236){1.48}}
\usefont{T1}{ptm}{m}{n}
\rput{0.09976467}(0.0,-0.0018415959){\rput(1.0574164,-0.27891734){$-\frac{m}{k-1}$}}
\usefont{T1}{ptm}{m}{n}
\rput{0.09976467}(0.0,-0.0069605396){\rput(3.857324,-0.23257938){$m+1$}}
\end{pspicture} 
}
\caption{Region $R \subset \CC$.}
\label{Region R}
\end{figure}
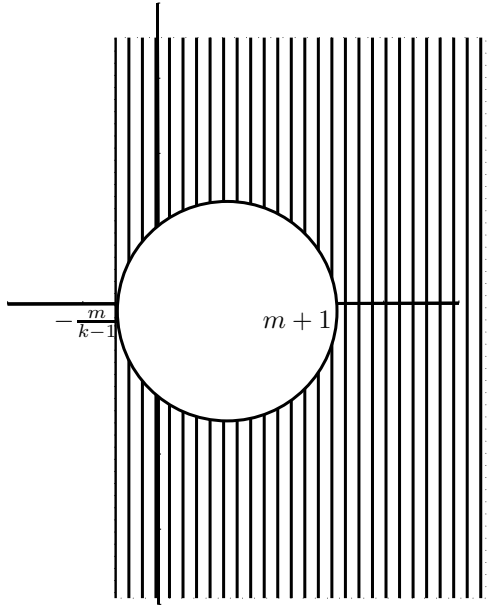

\end{center}

In fact we can even be more precise, depending on the regularity of $F$ (see Section $4$ for the definition of regular and Figure $2$ for the precise region we obtain).

\begin{theorem}\label{irregular}
Let $F$ be a germ of a holomorphic self-map of $\CC^2$ tangent to the identity. Assume $[v]$ is  a degenerate characteristic direction that is irregular. Then there exists an open basin attracted to the origin along $[v]$.
\end{theorem}

If $[v]$ is an apparent singularity, then as we see above, the index is always $0$ and the existence (or non-existence) of basins depends on the higher order terms of $F$ (i.e. not only on $(p_k,q_k)$ but also on $(p_j,q_j)$ for $j > k$).

We investigate this case in a subsequent paper.
%

\begin{remark}
In the case of the origin being dicritical, then we have $r(u) \equiv 0$. We also have $p_k(1,u) \not \equiv 0$. (If $p_k \equiv 0$ and the origin is dicritical, then we have $q_k \equiv 0$, but that is not possible since we are assuming $P_k \neq 0$.) Therefore we have:
\begin{align*}
k(u) = \lim_{z\to 0}\frac{z^kp_k(1,u) + O(z^{k+1})}{O(z^{k+1})} = \infty.
\end{align*}
In this case we say $\tilf$ is \textit{degenerate} along $\PP^1$. Abate also calls $F$ \textit{degenerate} along $\PP^1$ in \cite{Ab1}, and it is also called \textit{non-tangential} to $\PP^1$ in \cite{Ab-Tov}. 
\end{remark}



\section{\textbf{Conjugacy to the translation}}

We first prove a lemma which is a generalization of Hakim's theorem.

\begin{lemma}
Let $f=(f_1,f_2)$ germ of $(\CC^2,O)$ of the following form:
\[ 
\left\{ \begin{array}{l} f_1(z,w) = z + z^{a+1}w^b[c + O(z,w)]  \\ f_2(z,w) = w + z^aw^{b+1}[d +O(z,w)] 
\end{array}\right. 
\]
with $c \neq 0$, $d \neq 0$ and $a+b\geq 1$; $a,b$ non-negative.
If $c/d$ is such that:
\[
\Ree(c/d) > -\frac{b}{a} \qquad \textrm{and} \qquad \left|\frac{c}{d}+\frac{b}{2a}\right| > \frac{b}{2a}
\]
then the map has a basin attracted to the origin. In this basin the map is conjugate to a translation $(x,y) \to (x+1,y)$.
\end{lemma}

\begin{proof} 
This lemma is a result of Hakim in the case $b=0$. For $b>0$ we make a change of coordinates and transform our map into a germ of $(\CC^3,O)$ where we can apply Hakim's result, as follows.

We introduce $u = z^aw^b$. Then our map can be written as:
\begin{align*}
u_1 &= u + u^2(ac + bd +O(z,w))\\
z_1 &= z + zu(c + O(z,w))\\
w_1 &= w + wu(d + O(z,w))
\end{align*}
We can think of this map as a germ from $\CC^3$ to itself fixing the origin. The characteristic direction $(1,0,0)$ is not degenerate. We compute Hakim's index, which in this case is a matrix. We have:
$$
P_2(u,z,w) = ((ac+bd)u^2,czu,dwu) = (p_2(u,z,w),q_2(u,z,w))
$$
with $p_2 \in \CC$ and $q_2 \in\CC^2$.
Then in Hakim's formula we get (see \cite{Hak}):
$$
r(u_1,u_2) = q_2(1,u_1,u_2) - (u_1,u_2)p_2(1,u_1,u_2) = (cu_1,du_2) - (ac+bd)(u_1,u_2)
$$
therefore
$$
Dr(u_1,u_2)  =\left(\begin{array}{cc}(1-a)c-bd & 0 \\ 0 & -ac+(1-b)d\end{array}\right).
$$
If $\displaystyle{\Ree\left(\frac{(1-a)c-bd}{ac+bd}\right)>0}$ and  $\displaystyle{\Ree\left(\frac{(1-b)d-ac}{ac+bd}\right)>0}$ then we have a basin. We can see that this condition is equivalent to
\[
\Ree\left(\frac{c}{d}\right) > -\frac{b}{a} \qquad \textrm{and} \qquad \left|\frac{c}{d}+\frac{b}{2a}\right| > \frac{b}{2a}.
\]
Therefore we have a basis in $\CC^3$ and applying Hakim's theorem, we know that it can be conjugated to a translation $(x,y,t) \to (x,y,t+1)$. We can intersect this basis with $u = z^aw^b$ and we project in $(z,w)$ and we will get the required basin in $\CC^2$.
\end{proof}

The following lemma will be used repeatedly in the proof of Theorem 2. The proof is rather technical so we divide it into a series of steps.

\begin{lemma}[Main Lemma]\label{malem}
Let $(x_1,y_1) = G(x,y)$ be defined as:
\begin{align*}
x_1 &= x + 1 + \eta_1(x,y)\\
y_1 &= y + \frac{1}{x} + \eta_2(x,y)
\end{align*}
such that 
$$\eta_1(x,y) = O\left(\frac{1}{x^a},\frac{1}{y^b}\right)\textrm{ and }\eta_2(x,y) = O\left(\frac{y^c}{x^d},\frac{1}{xy^e}\right)\textrm{ with }a>0,b>0, c>0, d>1, e>0.
$$ 
Then there exists a domain $V$ in $\CC^2$ and an injective holomorphic map $\phi:V\to \CC^2$ such that:
\begin{itemize}
\item[(1)] $(x,y) \in V$ then $G(x,y) \in V$; and 
\item[(2)] $\phi\circ G\circ\phi^{-1}(u,v) = (u+1,v)$.
\end{itemize}
\end{lemma}

\begin{proof}
We will prove the lemma in several steps. First we find $V$, and then we will find $\phi$ as a composition of several change of variables.

Define 
\begin{align}\label{regionV}
V = V_{R,N,\theta} = \{(x,y) \in \CC^2: \Ree(x) > R, |\Arg(x)| <\theta, \Ree(y) > R, |y|^N < |x|\}
\end{align}
We prove that for a large $R, N$ and $\theta$, then $V$ is invariant under $G$.

Choose $N$ and $R$ large enough, so that $d - c/N>1$, so $|O(y^c/x^d)| < |O(x^{c/N}/x^d)| = |O(1/x^{d-c/N})|$ and $|\eta_1(x,y)| < 1/10$. Choosing $\theta << \pi/4$  we also have that $|\eta_2(x,y)| < 1/(10|x|)$ for all $(x,y) \in V$.

We prove now that if $(x,y) \in V$, then $(x_1,y_1) \in V$.

For the real part of $x_1$:
\begin{align*}
\Ree (x_1) = \Ree (x + 1) + \Ree (\eta_1) \geq \Ree x + 1 - |\eta_1| > \Ree x + \frac{9}{10} > R + \frac{9}{10} > R.
\end{align*}
We compute the argument of $x_1$:
\begin{align*}
|\tan \Arg (x_1)| \leq \frac{\Ree(x)|\tan \Arg(x)| + 1/10}{\Ree(x) + 11/10} 
\end{align*}
If $|\tan \Arg(x)|>1/11$, then we have 
\begin{align*}
|\tan \Arg (x_1)| \leq \frac{\Ree(x)|\tan \Arg(x)| + 1/10}{\Ree(x) + 11/10} < |\tan \Arg(x)|
\end{align*}
and therefore $|\Arg(x_1)| < |\Arg(x)| < \pi/4$.
If $0 \leq |\tan \Arg(x)| \leq 1/11$, then we have 
\begin{align*}
|\tan \Arg (x_1)| \leq \frac{\Ree(x)|\tan \Arg(x)| + 1/10}{\Ree(x) + 11/10} <  \frac{\Ree(x)/11 + 1/10}{\Ree(x) + 11/10} < 1/11
\end{align*}
and therefore $|\Arg(x_1)| < \pi/4$.
For the real part of $y_1$:
\begin{align*}
\Ree (y_1) = \Ree (y + \frac{1}{x}(1+x\eta_2)) = \Ree(y) + \Ree (\frac{1}{x}(1+x\eta_2)) 
\end{align*}
and
\begin{align*}
\Ree (\frac{1}{x}(1+x\eta_2)) &= \Ree\left(\frac{1}{x}\right)\Ree(1+x\eta_2) - \textrm{Im}\left(\frac{1}{x}\right)\textrm{Im}(1+x\eta_2)\\
&\geq \Ree\left(\frac{1}{x}\right)\Ree(1+x\eta_2) - \textrm{Re}\left(\frac{1}{x}\right)\textrm{Im}(1+x\eta_2) \\
&= \Ree\left(\frac{1}{x}\right) [\Ree(1+x\eta_2) - \textrm{Im}(1+x\eta_2)]\\
&= \Ree\left(\frac{1}{x}\right) [1+ \Ree(x\eta_2) - \textrm{Im}(x\eta_2)] > 8/10  \Ree\left(\frac{1}{x}\right) >0.
\end{align*}
The last point: $|y_1|^N < |x_1|$.
We have from our estimates above:
\begin{align*}
|x_1| &= |x|\left|1 + \frac{1}{x} + \frac{\eta_1(x,y)}{x}\right|\\
|y_1| &= |y|\left|1 + \frac{1}{xy} + \frac{\eta_2(x,y)}{y}\right|
\end{align*}
We use the estimates:
\begin{align*}
\left|1 + \frac{1}{x} + \frac{\eta_1(x,y)}{x}\right| \geq \frac{1}{|x|}\left[|x| + \frac{2}{5}\right] = 1 + \frac{2}{5|x|}
\end{align*}
and
\begin{align*}
\left|1 + \frac{1}{xy} + \frac{\eta_2(x,y)}{y}\right|^N < 1 + \frac{22N}{10|xy|}.
\end{align*}
Therefore:
\begin{align*}
|y_1|^N &= |y|^N\left|1 + \frac{1}{xy} + \frac{\eta_2(x,y)}{y}\right|^N < |y|^N\left(1 + \frac{22N}{10|xy|}\right) < |x|\left(1 + \frac{2}{5|x|}\right) = |x_1|
\end{align*}
if we choose $|y|> 22N/4$.
So, we have $V$ is invariant under the action of $G$.

Now we will prove the second part of the lemma.

We will separate some parts of the error terms and we'll deal with them first.
\begin{align*}
\eta_1(x,y) = \mu_1(y) + \rho_1(x,y)\textrm{ and } x\eta_2(x,y) = \mu_2(y) + \rho_2(x,y)
\end{align*}
where we separate the terms that contain only pure $y$ terms in both $\eta_1$ and $x\eta_2$. Therefore we have:
$$
\rho_1(x,y) = O\left(\frac{1}{x^a}\right)\textrm{ and }\rho_2(x,y) = O\left(\frac{y^c}{x^{d-1}}\right) < O\left(\frac{1}{x^{d-1-c/N}}\right)
$$
and
$$
\mu_1(y) = \left(\frac{1}{y^b}\right)\textrm{ and } \mu_2(y) = \left(\frac{1}{y^e}\right).
$$
Define $f(y)$ an analytic solution of the differential equation in the projection of $V$ to the second coordinate, for:
\begin{align}
(1+\mu_2(y))f'(y) + (1+ \mu_1(y))f(y)  = 1.
\end{align}

We see that 
\begin{align}\label{sizef}
f(y) = 1 + O(\mu_1,\mu_2) = O\left(\frac{1}{y^t}\right),
\end{align}
with $t>0$.

The first change of variables we make is the following:
\begin{align*}
(u,y) = \phi_1(x,y) = (xf(y),y) 
\end{align*}
Clearly $\phi_1$ is injective, because of \eqref{sizef}. Let us compute $(u_1,v_1) = \phi\circ G\circ\phi^{-1}(u,v)$:
\begin{align*}
u_1 &= x_1f(y_1) = \left(x+1+\mu_1(y)+\rho_1(x,y)\right)f\left(y + \frac{1+\mu_2(y)}{x} + \frac{\rho_2(x,y)}{x}\right)\\
&= \left(x+1+\mu_1(y)+\rho_1(x,y)\right)\left(f(y) + f'(y)\frac{1+\mu_2(y)}{x} +  O\left(\frac{1}{x^{1+\epsilon}y^{1+t}}\right)\right)\\
&=u+1+O\left(\frac{1}{u^{\epsilon}}\right)
\end{align*}
and
\begin{align*}
y_1 &= y + \frac{1+\mu_2(y)}{x} + \frac{\rho_2(x,y)}{xy}\\
&= y + \frac{1+\mu_2(y)}{u}- \frac{1+\mu_2(y)}{u} + \frac{1+\mu_2(y)}{x} + \frac{\rho_2(x,y)}{xy}\\
&= y + \frac{1+\mu_2(y)}{u}- \frac{1+\mu_2(y)}{u} + \frac{1+\mu_2(y)}{x} + \frac{\rho_2(x,y)}{xy}\\
&=y + \frac{1}{u} + O\left(\frac{1}{y^tu},\frac{1}{y^eu},\frac{1}{u^{1+\epsilon}}\right)
\end{align*}

Therefore, after conjugating $G$ we obtain $(u_1,y_1) = G_1(u,y) = \phi\circ G \circ \phi^{-1}(u,y)$:
\begin{align}
u_1 &= u+1+\kappa_1(u,y),\\
y_1 &= y + \frac{1}{u} + \kappa_2(u,y).
\end{align}
where $\displaystyle{\kappa_1(u,y) = O\left(\frac{1}{u^{\epsilon}}\right)}$ and $\displaystyle{\kappa_2(u,y) =  O\left(\frac{1}{y^tu},\frac{1}{y^eu},\frac{1}{u^{1+\epsilon}}\right)}$.
Let
\begin{align*}
g(u,y) = \int (1+\kappa_1(u,y))^{-1}du
\end{align*}
be any indefinite integral of $1/(1+\kappa_1)$ in the region $V'$. We can check easily that $\displaystyle{g(u,y) = u + O\left(\frac{1}{u^{\epsilon}}\right)}$. 
Define the change of variables as follows:
\begin{align*}
(s,y) = \phi_2(u,y) = (g(u,y),y)
\end{align*}
We have $\phi_2$ injective and we compute $(s_1,y_1) = \phi_2 \circ G_1 \circ \phi_2^{-1}(s,y)$.
\begin{align*}
s_1 &= g(u_1,y_1) = g(u,y) + (u_1-u)g_u(u,y) +(y_1-y)g_y(u,y) + O\left(g_uu,\frac{1}{u}\right)\\
&= s + \left(1+\kappa_1(u,y)\right)g_u(u,y) + O\left(\frac{1}{u}\right)\\
&= s + 1 + O\left(\frac{1}{u}\right).
\end{align*}
and
\begin{align*}
y_1 &= y + \frac{1}{s} - \frac{1}{s} + \frac{1}{u} + \kappa_2(u,y)\\
&=y +\frac{1}{s} + O\left(\frac{1}{y^\tau s},\frac{1}{s^{1+\epsilon}}\right).
\end{align*}
We can therefore write $(s_1,y_1) = G_2(s,y) =\phi_2 \circ G_1 \circ \phi_2^{-1}(s,y)$ as follows:
\begin{align}
s_1&= s + 1 + O\left(\frac{1}{s}\right)\\
y_1&=y +\frac{1}{s} + O\left(\frac{1}{y^\tau s},\frac{1}{s^{1+\epsilon}}\right). 
\end{align}
We repeat the same procedure to get $(t_1,y_1) = G_3(t,y) = \phi_3\circ G_2 \circ \phi_3^{-1}(t,y)$:
\begin{align}
t_1&= t + 1 + O\left(\frac{1}{t^2}\right)\\
y_1&=y +\frac{1}{t} + O\left(\frac{1}{y^\tau t},\frac{1}{t^{1+\epsilon}}\right). 
\end{align}
Let $\gamma(t,y) = t_1-t-1 = O\left(\frac{1}{t^2}\right)$. For any $(t_0,y_0) \in V''$, then we have $\sum_i\gamma(t_i,y_i)$ is bounded.
Therefore we can define the transformation $(z,y) = \phi_4(t,y) = (t + \sum_{i\geq 0}\gamma(t_i,y_i),y)$. We define $(z_1,y_1) = G_4(z,y) = \phi_4\circ G_3 \circ \phi_4^{-1}(s,y)$, so we obtain
\begin{align*}
(z_1,y_1) &= \phi_3(t_1,y_1) = (t_1 + \sum_{i\geq 1}\gamma(t_i,y_i),y_1)\\
&= (t+1+ \gamma(t,y) + \sum_{i\geq 1}\gamma(t_i,y_i),y_1)\\
&= (t+1+\sum_{i\geq 0}\gamma(t_i,y_i),y_1) = (z+1,y_1),
\end{align*}
where
\begin{align*}
y_1 = y + \frac{1}{t} + O\left(\frac{1}{y^\tau t},\frac{1}{t^{1+\epsilon}}\right)\\
= y+ \frac{1}{z} + O\left(\frac{1}{y^\tau z},\frac{1}{z^{1+\epsilon}}\right).\\
\end{align*}
Now we have the transformation $(z_1,y_1) = G_4(z,y)$:
\begin{align*}
z_1 &= z+1,\\
y_1 &= y+ \frac{1}{z} + \varpi(y,z),
\end{align*}
where $\displaystyle{\varpi(y,z) = O\left(\frac{1}{y^\tau z},\frac{1}{z^{1+\epsilon}}\right)}$. 
We need one more transformation in $y$. We need to separate some terms in $z\varpi(y,z) = O\left(1/y^\tau,1/z^{\epsilon}\right) = \Upsilon(y)+ \varpi_1(y,z)$ where $\Upsilon(y) = O(1/y^\tau)$ and $\varpi_1(y,z) = O(1/z^{\epsilon})$.
Define:
\begin{align*}
h(y) = \int(1+\Upsilon(y))^{-1}dy.
\end{align*}
Then:
\begin{align*}
h(y_1) &= h(y) + \left(\frac{1+\Upsilon(y)}{z} + O(1/z^{1+\epsilon})\right)h'(y) + O(1/z^2)\\
&= h(y) + \frac{1}{z} + \chi(z,y)
\end{align*}
where $\displaystyle{\chi(z,y) = O\left(\frac{1}{z^{1+\epsilon}}\right)}$.
If we write $(z,w) = G_5(z,y) = \phi_5 \circ G_4 \circ \phi_5^{-1}(z,y)$, for $\phi_5(z,y) = (z,h(y))$, then
\begin{align*}
z_1 &= z+1\\
w_1 &= w + \frac{1}{z} + \chi_1(z,w)
\end{align*}
where $\chi_1(z,w) = O(1/z^{1+\epsilon})$.

We notice that for $(z_0,w_0) \in V''$ then $\sum_{i\geq 0}\chi_1(z_i,w_i)$ is bounded. Therefore we can define:
$$
p(z_0,y_0) = \sum_{i\geq 0}\chi_1(z_i,w_i)
$$
and for the transformation $\phi_6(z,w) = (z,w+p(z,w)))$ we will have $(z_1,u_1) = G_6(z,u) = \phi_6 \circ G_5 \circ \phi_6^{-1}(z,u)$; we have:
\begin{align*}
z_1 = z+1,\qquad u_1 = u +\frac{1}{z}.
\end{align*}
in a region of the form $V$, as in \eqref{regionV} for appropriate $R, N, \theta$.
The following lemma concludes the proof of Lemma \ref{malem}.
\end{proof}

\begin{lemma}
If $(x_1,y_1) = F(x,y)$ of the following form:
\begin{align*}
x_1 &= x + 1\\
y_1 &= y + \frac{1}{x}
\end{align*}
in a region of the type $V$ as above, then there exists a change of coordinates $\phi(x,y)$ that transforms the region $V$ to a region of the same type (with possibly larger $N$ and $R$)
and such that $(z_1,w_1) = \phi \circ F \circ\phi^{-1}(z,w) = (z+1,w)$.
\end{lemma}

\begin{proof}
Use the change of coordinates $(z,u) = (x,y-\ln(x))$ where we choose a branch of the logarithm. Then we can compute:
\begin{align*}
u_1 &= y_1 - \ln(z_1) = y + \frac{1}{x} - \ln(x+1) = u + \ln(x) + \frac{1}{x} -\ln(x+1)\\
&= u + \psi(x)
\end{align*}
where $\psi(x) = O(1/x^2)$. Then we can do the usual change: $w_0 = u_0 + \sum_{i\geq 0}(\psi(x_i))$ and we obtain $w_1 = u_1 + \sum_{i\geq 1}(\psi(x_i)) = u_0 + \psi(x_0) + \sum_{i\geq 1}(\psi(x_i)) = u_0 + \sum_{i\geq 0}(\psi(x_i)) = w_0$.
So, we obtain the desired conjugacy.
Note that the condition was $|y|^N<|x|$ and this translates into $|w|^{2N} < |x|$, since for large $N$ and $R$ we have $|\ln(x)| < |x|^{1/N}$ for any $N$.
\end{proof}

\section{\textbf{Fuchsian Singularities}}

In this section we will prove Theorem \ref{Fuchsian}. The strategy is similar to the one above. We divide in different cases, change coordinates and then apply the lemmas proven in the last section.

\begin{proof}[Proof of Theorem \ref{Fuchsian}]:
We have two cases: either $m=0$ (in which case we have a non-degenerate characteristic direction) or $m>0$ (degenerate characteristic direction).

\subsection{\textbf{Case (a.1): m$=$0, n$=$1}}

\begin{align*} 
\left\{ \begin{array}{l} \tilde{F}_1(z,u) = z + z^{k}[a_0 + O(z,u)]  \\ \tilde{F}_2(z,u) = u + z^{k-1}[c_{1}u + O(z,u^2)] 
\end{array}\right. 
\end{align*}
and we know:
\begin{itemize}
\item $a_0 \neq 0$ and $c_1 \neq 0$ 
and							
\item $\Ind(\tilde{F},\PP^1,[v]) = a_0/c_1$
\end{itemize}
where $c_1 = b_1 - a_0$.

\begin{proposition} Let $\tilde{F}$ be as above. Then if $\Ree(\Ind(\tilde{F},\PP^1,[v]))>0$ there exists an open basin for $\tilde{F}$.
\end{proposition}

\begin{proof}
This has already been proved by Hakim \cite{Hak}. 
\end{proof}

\subsection{\textbf{Case (a.2): m$>$0, n$=$m+1}}

\begin{align*} 
\left\{ \begin{array}{l} \tilde{F}_1(z,u) = z + z^{k}[a_mu^m + O(z,u^{m+1})]  \\ \tilde{F}_2(z,u) = u + z^{k-1}[c_{m+1}u^{m+1} + O(z,u^{m+2})] 
\end{array}\right. 
\end{align*}

\begin{proposition}
Given $\tilde{F}$ as above. If $\Ind(\tilde(F),\PP^1,[v]) = \zeta \in R$, where 
$$
R = \left\{\zeta\in\CC : \Ree(\zeta) > -\frac{m}{k-1}, \left|\zeta-\frac{m+1+m/(k-1)}{2}\right| > \frac{m+1+m/(k-1)}{2}\right\}
$$
then $\tilde{F}$ has a basin attracted to the origin.
\end{proposition}

\begin{proof}
We use the change of variables: 
\[ x = \frac{z}{u^{m+1}} \qquad u = u.\]

In these coordinates we obtain
\[ 
\left\{ \begin{array}{l} x_1 = x + x^ku^{mk+k-1}[(1-(m+1)\beta) + O(x,u)]\\
u_1 = u + x^{k-1}u^{mk+k}[\beta + O(x,u)] 
\end{array}\right. 
\]
where $\beta=\frac{1}{\Ind(\tilde{F},\PP^1,[v])}$
We apply lemma 1, which was proved in the last section.

\begin{lemma}
Let $F=(f_1,f_2)$ where
\[ 
\left\{ \begin{array}{l} f_1(z,w) = z + z^{a+1}w^b[c + O(z,w)]  \\ f_2(z,w) = w + z^aw^{b+1}[d +O(z,w)] 
\end{array}\right. 
\]
If $c/d$ is such that:
\[
\Ree(c/d) > -\frac{b}{a} \qquad \textrm{and} \qquad \left|\frac{c}{d}+\frac{b}{2a}\right| > \frac{b}{2a}
\]
then the map has a basin attracted to the origin. In this basin the map is conjugated to a translation $(x,y) \to (x+1,y)$.
\end{lemma}

In our case,
$$
a = k-1, b = mk+k-1, c = 1 - (m+1)\beta, d = \beta,
$$
and therefore:
\begin{eqnarray*}
\Ree\left(\frac{1}{\beta} - (m+1)\right)> - \frac{mk+k-1}{k-1}\textrm{  and  } \left|\frac{1}{\beta} - (m+1)+\frac{mk+k-1}{2(k-1)}\right| > \frac{mk+k-1}{2(k-1)}.
\end{eqnarray*}
This in turn becomes
\begin{eqnarray*}
\Ree\left(\frac{1}{\beta}\right)> -\frac{m}{k-1},
\end{eqnarray*}
and
\begin{eqnarray*}
\left|\frac{1}{\beta} - (m+1)+\frac{mk+k-1}{2(k-1)}\right| > \frac{mk+k-1}{2(k-1)},
\end{eqnarray*}
which is exactly the region:
$$\left\{\zeta \in \CC, \Ree(\zeta) > -\frac{m}{k-1}, \left|\zeta-\frac{m+1+m/(k-1)}{2}\right| > \frac{m+1+m/(k-1)}{2}\right\}.
$$
\end{proof}
We can say a little more about other regions in $\CC$ for which there will be a basin also. Recall the expression of $\tilde{F}$:
\begin{align*} 
\left\{ \begin{array}{l} \tilde{F}_1(z,u) = z + z^{k}[u^m + O(z,u^{m+1})]  \\ \tilde{F}_2(z,u) = u + z^{k-1}[\beta u^{m+1} + \rho z + O(z^2,zu,u^{m+2})] 
\end{array}\right. 
\end{align*}
We say $\tilde{F}$ is \textit{regular} if $\rho \neq 0$ (following Molino's terminology \cite{Mo}).
\begin{itemize}
\item[A.] $\tilde{F}$ is regular.
We change variables:
$$
(t,u) = \phi(z,u) = \left(\frac{z}{u^m},u\right)
$$
and we get $(t_1,u_1) = G(t,u) = \phi\circ\tilde{F}\circ\phi^{-1} (t,u)$:

\[ 
\left\{ \begin{array}{l} t_1 = t + t^{k-1}u^{mk-1}[ - m\rho t^2  + (1-m\beta)tu + O(t^2u,tu^2)]
\\ u_1 = u + t^{k-1}u^{mk-1}[\rho tu + \beta u^2 + O(tu^2,u^3)]
\end{array}\right. 
\]
Then we have one non degenerate characteristic direction: 
$$(1-(m+1)\beta,(m+1)\rho).$$
The Hakim index for this non degenerate characteristic direction is
$$
-(m+1)\left(1-(m+1)\beta\right).
$$
Using Hakim's theorem we know that if $\Ree\left(-(m+1)\left(1-(m+1)\beta\right)\right)>0$ we have a basin. Unraveling, we obtain:
$$
\Ree(\beta) = \Ree\left(\frac{1}{\Ind(\tilde{F},\PP^1,[v])}\right) > \frac{1}{m+1},
$$
will guarantee the existence of a basin.
Therefore $\Ind(\tilde{F},\PP^1,[v])$ is in the region:
$$
R_1 = \left\{\left|\zeta - \frac{1}{2(m+1)}\right| < \frac{1}{2(m+1)} \right\}.
$$
\\
\item[B.] $\tilde{F}$ is not regular.
Then $\rho = 0$ and in the same change of variables we obtain:
\[ 
\left\{ \begin{array}{l} t_1 = t + t^{k}u^{mk}[(1-m\beta) + O(t,u)]
\\ u_1 = u + t^{k-1}u^{mk+1}[\beta + O(t,u)]
\end{array}\right. 
\]
We can apply Lemma 1 again, and therefore we get: if $\Ind(\tilde{F},\PP^1,[v])$ is in the following region
$$
R_2=\left\{\zeta \in \CC, \Ree(\zeta) > -\frac{m}{k-1}, \left|\zeta-\frac{m+m/(k-1)}{2}\right| > \frac{m+m/(k-1)}{2}\right\}
$$
then we do have a basin. 
\end{itemize}

Let us summarize. We do have a basin for $\Ind(\tilde{F},\PP^1,[v]) \in R \cup S$ where $S = R_1 \cap R_2$. If $\tilde{F}$ is regular then we have a basin for $\Ind(\tilde{F},\PP^1,[v]) \in R \cup R_1$ and if $\tilde{F}$ is not regular we have a basin for $\Ind(\tilde{F},\PP^1,[v]) \in R \cup R_2$.

\begin{center}
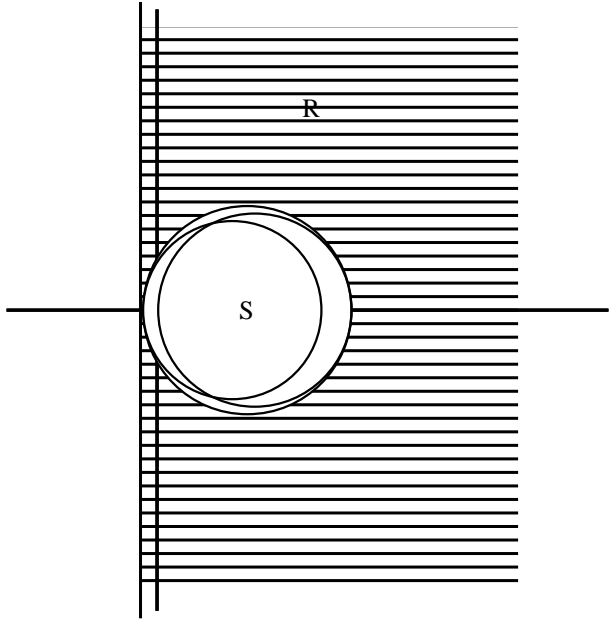
\begin{figure}[h]
{
\begin{pspicture}(0,-4.12)(8.0,4.12)
\psframe[linewidth=0.02,linecolor=white,dimen=outer,fillstyle=hlines*,hatchwidth=0.04,hatchangle=0.0](6.82,3.78)(1.78,-3.78)
\rput(2.0,0.0){\psaxes[linewidth=0.04,labels=none,ticks=none,ticksize=0.10583333cm](0,0)(-2,-4)(6,4)}
\pscircle[linewidth=0.03,dimen=outer,fillstyle=solid](3.2,0.0){1.4}
\pscircle[linewidth=0.03,dimen=outer,fillstyle=solid](3.0,0.0){1.2}
\pscircle[linewidth=0.03,dimen=outer](3.3,0.0){1.3}
\usefont{T1}{ptm}{m}{n}
\rput(3.19,0.0050){S}
\usefont{T1}{ptm}{m}{n}
\rput(4.05,2.685){R}
\psline[linewidth=0.04cm](1.78,4.1)(1.78,-4.1)
\end{pspicture} 
}
\caption{Region $R$ and $S$.}
\end{figure}
\end{center}

\end{proof}

\begin{remark}
When $m=0$ the region $R \cup S$  is the whole right plane (minus the circle around $S$), which is Hakim's result.
\end{remark}

\section{\textbf{Irregular characteristic directions}}

Here we prove Theorem \ref{irregular}. We divide it into the cases above and apply a change of variables. After that we apply the lemmas proven above.
 
\begin{proof}[Proof of Theorem \ref{irregular}]:
We will prove that there is a basin for $\tilde{F}$ and therefore for $F$.
Recall \eqref{blowF}:
\begin{align*}
z_1 &= z + z^k[p_k(1,u) + O(z)]\\
u_1 &= u + z^{k-1}[r(u) + O(z)]
\end{align*}

We divide in two cases: $m=0$ (as in Case (b.1)) and $m>0$ (as in Case (b.2)).

\subsection{\textbf{Case (b.1): m=0, n$>$1}}
Then we have:
\begin{align*}
z_1 &= z + z^k[a_0 + O(z,u)]\\
u_1 &= u + z^{k-1}[c_nu^n + O(z,u^{n+1})]
\end{align*}
with $a_0 \neq 0$ and $c_n  = b_n - a_{n-1}\neq 0$. Using a linear change of coordinates for $z$ we can assume $a_0 = -1$, and similarly for $u$ we assume $c_n = -1$.

Use the transformation
\[ 
x = \frac{1}{(k-1)z^{k-1}}, \qquad y = \frac{k-1}{(n-1)u^{n-1}}
\]
on a suitable open set, with the origin on its boundary. 

In these coordinates we have:
\begin{eqnarray}
x_1 &=& x + 1 + O\left(\frac{1}{x^{1/(k-1)}},\frac{1}{y^{1/(n-1)}}\right)\\
y_1 &=& y + \frac{1}{x} + O\left(\frac{y^{1/(n-1)}}{x^{k/(k-1)}},\frac{1}{xy^{n/(n-1)}}\right)
\end{eqnarray}	

We will show how we got the expression for $y_1$ (the expression for $x_1$ is immediate). 
\begin{eqnarray*}
u_1 &=& u + z^{k-1}(-u^n +O(z,u^{n+1}))\\
u_1^{n-1} &=& \left(u + z^{k-1}[-u^n +O(z,u^{n+1})]\right)^{n-1}\\
u_1^{n-1} &=& u^{n-1}\left(1 + \frac{z^{k-1}}{u}\left[-u^n +O(z,u^{n+1})\right]\right)^{n-1} \\
\frac{1}{y_1} &=& \frac{1}{y}\left(1 - (n-1)z^{k-1}u^{n-1} + O\left(\frac{z^k}{u},z^{k-1}u^n\right)\right)\\
y_1 &=& y\left(1 + (n-1)z^{k-1}u^{n-1} + O\left(\frac{z^k}{u},z^{k-1}u^n\right)\right)\\
y_1 &=& y\left(1 + (n-1)\frac{1}{(k-1)x}u^{n-1} + O\left(\frac{z^k}{u},z^{k-1}u^n\right)\right)\\
y_1 &=& y + \frac{1}{x} + O\left(\frac{y^{1/(n-1)}}{x^{k/(k-1)}},\frac{1}{xy^{n/(n-1)}}\right)
\end{eqnarray*}

We now apply Lemma \ref{malem}, which concludes the proof in this case.

\subsection{\textbf{Case (b.2): m$>$0, n$>$m+1}}

We have $a_{m} \neq 0$ and $a_{i} = 0$ for all $i<m$ and the analogous for $c_{j}$. 

Without loss of generality we assume $a_{m}=-1$ and $c_{n}=-1$. 

\begin{align*}
z_1 &= z + z^k[-u^m + O(z,u^{m+1})]\\
u_1 &= u + z^{k-1}[-u^n + O(z,u^{n+1})]
\end{align*}

We use the following change of coordinates:
\[ 
x = \frac{k-1}{z^{k-1}u^m}, \qquad y = \frac{(k-1)(n-m-1)}{u^{n-m-1}}
\]
In these coordinates we have:
\begin{eqnarray}
x_1 &=& x + 1  + O\left(\frac{1}{y^{1/(n-m-1)}},\frac{y^{\frac{km+k-1}{(k-1)(n-m-1)}}}{x^{1/(k-1)}}\right)\\
y_1 &=& y + \frac{1}{x} + O\left(\frac{y^{\frac{n-m+k-2+mk}{(n-m-1)(k-1)}}}{x^{k/(k-1)}},\frac{1}{xy^{1/(n-m-1)}}\right)
\end{eqnarray}
where the powers are chosen as a branch on a suitable open set.

And once again we apply Lemma \ref{malem}.
\end{proof}

\begin{remark}
In her paper \cite{Mo}, Molino proves that $(1,\alpha)$ is a non-degenerate characteristic direction for the map:
\begin{align*}
z_1 &= z + z^k[-1+ O(z,u)]\\
u_1 &= u + z^{k-1}[c_nu^n - \alpha z +  O(z^2,u^{n+1})],
\end{align*}
for the first case and also proves there exists a non-degenerate characteristic direction $(1,\frac{a_m}{(m+1)\alpha})$ for the map:
\begin{align*}
z_1 &= z + z^k[-u^m + O(z,u^{m+1})]\\
u_1 &= u + z^{k-1}[\alpha z -u^n + O(z^2,u^{n+1})]
\end{align*}
for the second case.
Therefore this proves that there exists a parabolic curve. Nonetheless there is no basin associated to these characteristic directions. It is an easy computation to show that the Hakim index associated to both is negative, which means that there is no basin along that direction.  
\end{remark}






\section{Basins as Fatou-Bieberbach domains}

Given an automorphism of $\CC^2$ with a fixed point (say the origin) and attractive (i.e. $dF(0)$ has only eigenvalues with modulus less than $1$) is a well-known fact that the basin associated to the fixed point, is biholomorphic to $\CC^2$ (therefore a so-called Fatou-Bieberbach domain).

If the automorphism is tangent to the identity, Hakim proved that the basin associated to the non-degenerate characteristic directions are also biholomorphic to $\CC^2$. 

We will prove in this section that, if the map tangent to the identity in Theorems $1$ and $2$ is an automorphism of $\CC^2$ then the basins are biholomorphic to $\CC^2$.

Proving that a basin is biholomorphic to $\CC^2$ is in some sense a local statement. If we find a region $V$ such that the map is conjugated  in $V$ to a translation
$\phi \circ F \circ \phi^{-1}(z,w) = (z+1,w)$
for $(z,w) \in W = \phi(V)$, then we can define a map from the entire basin $\Omega = \bigcup_{i \geq 0}F^{-i}(V)$ to $\CC^2$ as follows:
\begin{align*}
\Phi: &  \Omega \to \CC^2\\
\Phi(p)  =& \phi \circ F^n(p) - (n,0)
\end{align*}
for any $n$ such that $F^n(p) \in V$. It is standard to see that this map is well defined and independent of $n$.
Then we can easily check that $\Phi$ is injective, and therefore is a biholomorphism between $\Omega$ and its image $\Phi(\Omega)$.

Clearly $\Phi(\Omega) = \bigcup_{n \geq 0} W - (n,0)$. Therefore, to prove that $\Omega$ is biholomorphic to $\CC^2$ we have to prove that $\bigcup_{n \geq 0} W - (n,0)$ is all of $\CC^2$.

Recall now that our region $V'$, before the last change of coordinates, is of the form $V' = V_{R,N,\theta} = \{(x,y) \in \CC^2: \Ree(x) > R, |\Arg(x)| <\theta, \Ree(y) > R, |y|^N < |x|\}$ for some $R,N,\theta$ in the Lemma $3$. 

We then change coordinates as $(z,u) = (x,y-\ln (x)) = \psi(x,y)$, and our region $V'$ becomes $W = \psi(V') = \{(x,y) \in \CC^2: \Ree(x) > R, |\Arg(x)| <\theta, y \in \CC\}$. We clearly have $\bigcup_{n \geq 0} W - (n,0) = \CC^2$.

All the basins we encounter in the Theorems $1$ and $2$ are therefore biholomorphic to $\CC^2$, since all of them were, either conjugate to the translation in a region as above, or came from non-degenerate characteristic directions. In the latter case due to Hakim, we already have that they are biholomorphic to $\CC^2$.

%

\bibliography{biblio}

\end{document}